\documentclass[10pt,leqno]{amsart}
\usepackage{amssymb,amsfonts}
\usepackage{amsmath,amsthm,amsxtra}
\usepackage[all]{xy}
\usepackage{mathabx}
\usepackage{color}
\usepackage{verbatim}

\usepackage[margin=0.76in]{geometry}
\usepackage[linktocpage=true,colorlinks=true, linkcolor=blue, citecolor=red, urlcolor=green]{hyperref}

\usepackage{enumerate}

\newcommand\bigcheck[1]{#1 \raise1ex\hbox{$\hspace{-1ex}{}^\vee$}}
\newcommand\sucheck[1]{#1 \raise0.5ex\hbox{$\hspace{-1ex}{}^\vee$}}

\newtheorem{theorem}{Theorem}[section]

\newtheorem{corollary}[theorem]{Corollary}

\newtheorem*{lemma*}{Lemma}

\theoremstyle{definition}

\theoremstyle{remark}
\newtheorem{remark}[theorem]{Remark}
\newtheorem{example}[theorem]{Example}

\newcommand{\mc}[1]{{\mathcal #1}}

\newcommand{\mb}[1]{{\mathbb #1}}

\renewcommand{\tilde}{\widetilde}

\newcommand{\Mat}{\mathop{\rm Mat }}

\renewcommand{\ker}{\mathop{\rm Ker }}

\definecolor{light}{gray}{.9}

\begin{document}


\title{Some remarks on non-commutative principal ideal rings}


\author{
Sylvain Carpentier,
Alberto De Sole,
Victor G. Kac
}

\address{
Ecole Normale Superieure, Paris, France,
e-mail:  \rm \texttt{scarpent@clipper.ens.fr}
}

\address{
Dept. of Math., University of Rome 1, Italy
e-mail:  \rm \texttt{desole@mat.uniroma1.it}
}

\address{
Dept. of Math., MIT, USA
e-mail:  \rm \texttt{kac@math.mit.edu}
}

\maketitle


\noindent
{\bf Abstract:} We prove some algebraic results on the ring of matrix differential operators
over a differential field in the generality of non-commutative principal ideal rings.
These results are used in the theory of non-local Poisson structures.

\noindent 
{\bf R\'esum\'e:} Nous d\'emontrons quelques r\'esultats alg\'ebriques sur l'anneau des 
matrices op\'erateurs diff\'erentiels sur un corp diff\'erentiel 
dans la g\'en\'eralit\'e des anneaux non-commutatives principaux.
Ces r\'esultats sont utilis\'es dans la th\'eorie des structures de Poisson non-locales.

\section{Introduction}
\label{sec:intro}

In our previous two papers \cite{CDSK12a,CDSK12b}
we established some algebraic properties of the ring of matrix differential operators
over a differential field.
The problems naturally arose in the study of non-local Poisson structures \cite{DSK12a,DSK12b}.

Eventually we realized that the proofs of \cite{CDSK12b} can be simplified,
so that our results hold in the full generality of left and right principal ideal rings.

The new result which is not contained in our previous paper is Theorem \ref{20121124:thm2},
which is used in the theory of non-local Lenard-Magri scheme in \cite{DSK12b}.

We wish to thank Toby Stafford and Lance Small for very useful correspondence.
In particular, Toby Stafford provided us a proof of Theorem \ref{thm:list}.

\section{General facts about principal ideal rings}
\label{sec:2}

Let $R$ be a unital associative (not necessarily commutative) ring.
Recall that
a \emph{left} (resp. \emph{right}) \emph{ideal} of $R$
is an additive subgroup $I\subset R$ such that $RI=I$ (resp. $IR=I$).
The left (resp. right) \emph{principal} ideal generated by $a\in R$
is, by definition, $Ra$ (resp. $aR$).

Throughout the paper, we assume that the ring $R$ is both a left and a right
\emph{principal ideal ring},
meaning that every left ideal of $R$ and every right ideal of $R$ is principal.
\begin{example}\label{20121214:ex}
Let $\mc K$ be a differential field with a derivation $\partial$,
and let $\mc K[\partial]$ be the ring of differential operators over $\mc K$.
It is well known that $\mc K[\partial]$ is a left and right principal ideal domain, 
see e.g. \cite{CDSK12a}.
Let $\mc R=\Mat_{\ell\times\ell}(\mc K[\partial])$
be the ring of $\ell\times\ell$ matrices with coefficients in $\mc K[\partial]$.
By Theorem \ref{thm:list}(a) below, the ring $\mc R$ is a left and right principal ideal ring
as well.
Note also that $\mc K^\ell$ is naturally a left $\mc R$-module.
\end{example}

Given an element $a\in R$,
an element $d\in R$ is called a \emph{right} (resp. \emph{left}) \emph{divisor} of $a$
if $a=a_1d$ (resp. $a=da_1$) for some $a_1\in R$.
%
An element $m\in R$ is called a \emph{left} (resp. \emph{right}) \emph{multiple} of $a$
if $m=qa$ (resp. $m=aq$) for some $q\in R$.

Given elements $a,b\in R$,
their \emph{right} (resp. \emph{left}) \emph{greatest common divisor} 
is the generator $d$ of the \emph{left} (resp. \emph{right})
ideal generated by $a$ and $b$:
$Ra+Rb=Rd$ (resp. $aR+bR=dR$).
It is uniquely defined up to multiplication by an invertible element.
It follows that $d$ is a right (resp. left) divisor of both $a$ and $b$,
and we have the \emph{Bezout identity}
$d=ua+vb$ (resp. $d=au+bv$) for some $u,v\in R$.

Similarly, the \emph{left} (resp. \emph{right}) \emph{least common multiple}
of $a$ and $b$ is an element $m\in R$,
defined, uniquely up to multiplication by an invertible element,
as the generator of the intersection of the left (resp. right) principal ideals generated 
by $a$ and by $b$:
$Rm=Ra\cap Rb$
(resp. $mR=aR\cap bR$)..

We say that $a$ and $b$ are \emph{right} (resp. \emph{left}) \emph{coprime}
if their \emph{right} (resp. \emph{left}) greatest common divisor is $1$ (or invertible),
namely if the \emph{left} (resp. \emph{right}) ideal that they generate is the whole ring $R$:
$Ra+Rb=R$ (resp. $aR+bR=R$).
Clearly, this happens if and only if we have the Bezout identity $ua+vb=1$ (resp. $au+bv=1$)
for some $u,v\in R$.

An element $k\in R$ is called a \emph{right} (resp. \emph{left}) \emph{zero divisor} if
there exists $k_1\in R\backslash\{0\}$ such that $k_1k=0$ (resp. $kk_1=0$).
Note that, if $d$ is a right (resp. left) divisor of $a$,
and $d$ is a left (resp. right) zero divisor, then so is $a$.
In particular, if either $a$ or $b$ is not a left (resp. right) zero divisor,
then their right (resp. left) greatest common divisor $d$ is also
not a left (resp. right) zero divisor.
A non-zero element $b\in R$ is called \emph{regular} if it is neither a left nor a right zero divisor.

The following results summarize some important properties of principal ideal rings
that will be used in this paper.
Since a principal ideal ring is obviously Noetherian,
one can use the powerful theory of non-commutative Noetherian rings
(see \cite{MR01}).
\begin{theorem}\label{thm:list}
Let $R$ be a left and right principal ideal ring. Then:
\begin{enumerate}[(a)]
\item The ring $\Mat_{\ell\times\ell}(R)$ of $\ell\times\ell$ matrices with entries in $R$
is a left and right principal ideal ring.
\item
The sets of left and right zero divisors of $R$ coincide.
Hence, an element of $R$ is regular if and only if it is not a left (or a right)
zero divisor.
\item
The set of regular elements of $R$ satisfies the left (resp. right) Ore property:
for $a,b\in R$ with $b$ regular,
there exist $a_1,b_1\in R$, with $b_1$ regular, such that $ba_1=ab_1$
(resp. $a_1b=b_1a$).
\item
There exists the \emph{ring of fractions} $Q(R)$ containing $R$,
consisting of left fractions $ab^{-1}$
(or, equivalently, right fractions $b^{-1}a$),
with $a,b\in R$ and $b$ regular.
\item
Given $a,b\in R$ with $b$ regular,
there exists $q\in R$ such that $a+qb$ (resp. $a+bq$) is regular.
\item
Suppose that the ring $R$ contains a central regular element $r\in R$ such that $r-1$ 
is regular too.
Given $a_1,a_2,b_1,b_2\in R$ with $b_1,b_2$ regular,
there exists $q\in R$ such that $a_1+qb_1$ and $a_2+qb_2$ 
(resp. $a_1+b_1q$ and $a_2+b_2q$) are both regular.
\end{enumerate}
\end{theorem}
\begin{proof}[Proof (by J.T. Stafford)]
Statement (a) is in \cite[Prop.3.4.10]{MR01}.
For part (b) \cite[Cor.4.1.9]{MR01}
shows that $R$ is a direct sum  $R=A \oplus B$ of an Artinian ring A 
and a Noetherian semiprime ring $B$.
Obviously the regular elements of $A$ are just the units.
By  Goldie's Theorem the right regular elements of 
$B$ are the same as the left regular elements, i.e. the regular elements 
(see \cite[Prop.2.3.4 and 2.3.5]{MR01}). 
Since an element $(a,b) \in R=A\oplus B$ is regular if and only if 
$a$ and $b$ are both regular  the same conclusion holds for $R$.
This proves (b)
The equivalence of (c) and (d)  is Ore's Theorem \cite[Thm.2.1.12]{MR01}.
Part (c)  then follows from Goldie's Theorem.
It is routine to see that the regular elements of $A\oplus B$ 
form an Ore set if this is true for both $A$ and $B$. 
Of course this result is vacuously true for $A$ while Goldie's Theorem does it for $B$.
Part (e) follows from \cite[Cor.2.5]{SS82}, and part (f) is in \cite{Sta12}.
\end{proof}
\begin{remark}\label{20121217:rem}
As T. Stafford pointed out, the ring $R=\mb Z/2\mb Z$ does not satisfy the property in part (f).
\end{remark}

\begin{remark}\label{20121217:rem}
From the above theorem we immediately get the following simple observations.
\begin{enumerate}[(a)]
\item
By Theorem \ref{thm:list}(b) we have that if $a=bc$, 
then $a$ is regular if and only if $b$ and $c$ are regular. 
In particular, any left or right divisor of a regular element is regular.
\item
If $b$ is regular and $a$ arbitrary, then 
we can write their right (resp. left) least common multiple
as $ab_1=ba_1$ with $b_1$ regular.
This follows from the Ore property in Theorem \ref{thm:list}(c).
Indeed, let $I=\{b'\in R\,|\,ab'\in bR\}$. It is clearly a right ideal of $R$.
Hence, $I=b_1R$ for some $b_1$.
Clearly, $m=ab_1$ is the right least common multiple of $a$ and $b$.
By the Ore property, there exists a regular element $\tilde b\in I$.
Hence, $\tilde b=b_1c$, and therefore $b_1$ is regular too.
\item
It follows from the above observation that,
if $a$ and $b$ are regular, so is their right (resp. left) least common multiple.
\item
If $a=a_1d$, $b=b_1d$, (resp. $a=da_1$, $b=db_1$), 
and $a_1$ and $b_1$ are right (resp. left) coprime,
then $d$ is the right (resp. left) greatest common divisor of $a$ and $b$.
Indeed, 
by the Bezout identity we have $ua_1+vb_1=1$ (resp. $a_1u+b_1v=1$),
which implies $ua+vb=d$ (resp. $au+bv=d$).
But then $Rd=Ra+Rb$ (resp. $dR=aR+bR$), proving the claim.
\item
Conversely,
if $a=a_1d$, $b=b_1d$, (resp. $a=da_1$, $b=db_1$), 
and $d$ is the right (resp. left) greatest common divisor of $a$ and $b$,
then, assuming that $d$ is regular, we get that $a_1$ and $b_1$ are right (resp. left) coprime.
Indeed,
by the Bezout identity we have $d=ua+vb=(ua_1+vb_1)d$
(resp. $d=au+bv=d(a_1u+b_1v)$),
and since by assumption $d$ is regular it follows that $ua_1+vb_1=1$.
\end{enumerate}
\end{remark}

\section{Some arithmetic properties of principal ideal rings}
\label{sec:3}

\begin{theorem}\label{20121122:prop}
Let $R$ be a left and right principal ideal ring and let $Q(R)$ be its ring of fractions.
Let $f=ab^{-1}=a_1b_1^{-1}\in Q(R)$
(resp. $f=b^{-1}a=b_1^{-1}a_1\in Q(R)$), 
with $a,a_1,b,b_1\in R$ and $b,b_1$ regular,
and assume that $a_1$ and $b_1$ are right (resp. left) coprime.
Then there exists a regular element $q\in R$
such that $a=a_1q$ and $b=b_1q$ (resp. $a=qa_1$ and $b=qb_1$).
\end{theorem}
\begin{proof}
By assumption $a_1$ and $b_1$ are right coprime,
hence we have the Bezout identity 
$ua_1+vb_1=1$, for some $u,v\in R$.
Let $q=ua+vb$.
We have
$$
\begin{array}{l}
\vphantom{\Big(}
b_1q=b_1(ua+vb)=b_1(uab^{-1}+v)b
=b_1(ua_1b_1^{-1}+v)b
\\
\vphantom{\Big(}
=b_1(ua_1+vb_1)b_1^{-1}b=b\,,
\end{array}
$$
and
$$
a_1q=a_1b_1^{-1}b_1q=
a_1b_1^{-1}b=ab^{-1}b=a\,.
$$
Finally, $q$ is regular since $q=b_1^{-1}b$ is invertible in $Q(R)$.
\end{proof}
\begin{corollary}\label{20121122:cor}
For every $f\in Q(R)$ there is a ``minimal''
right (resp. left) fractional decomposition $f=ab^{-1}$ (resp. $f=b^{-1}a$)
with $a,b$ right (resp. left) coprime.
Any other right (resp. left) fractional decomposition 
is obtained from it by simultaneous multiplication of $a$ and $b$
on the right (resp. left) by some regular element $q\in R$.
\end{corollary}
\begin{proof}
It follows immediately from Remark \ref{20121217:rem}(d)
and Theorem \ref{20121122:prop}.
\end{proof}

\begin{theorem}\label{20121124:thm2}
Let $R$ be a left and right principal ideal ring,
and let $V$ be a left module over $R$.
Assume that the ring $R$ contains a central regular element $r\in R$ such that $r-1$ 
is regular too.
Let $a,b\in R$, with $b$ regular, be left coprime.
Let $m=ab_1=ba_1$ be their right least common multiple.
Then, for every $x,y\in V$ such that $ax=by$,
there exists $z\in V$ such that $x=b_1z$ and $y=a_1z$.
In particular, $aV\cap\ bV=mV$.
\end{theorem}
\begin{proof}
We first reduce to the case when $a$ is regular.
Indeed, let, by Theorem \ref{thm:list}(e), $q\in R$ be such that $a+bq$ is regular.
Then it is immediate to check that
the right least common multiple of $a+bq$ and $b$ is
$(a+bq)b_1=b(a_1+qb_1)$.
Moreover, since by assumption $ax=by$, we have $(a+bq)x=b(y+qx)$.
Therefore, assuming that the theorem holds for regular $a$,
there exists $z\in V$ such that
$x=b_1z$ and $y+qx=(a_1+qb_1)z$, which implies $y=a_1z$,
proving the claim.

Next, let us prove the theorem under the assumption that both $a$ and $b$ are regular.
Since $m=ab_1=ba_1$ is the right least common multiple of $a$ and $b$,
it follows that $a_1$ and $b_1$ are right coprime,
and therefore we have the Bezout identity
\begin{equation}\label{20121214:eq1}
ub_1+va_1=1\,,
\end{equation}
for some $u,v\in R$.
After replacing $u$ by $u+qa$ and $v$ by $v-qb$,
equation \eqref{20121214:eq1} still holds.
Hence, by Theorem \ref{thm:list}(f),
we can assume, without loss of generality, that $u$ and $v$ are both regular.
Moreover, by Remark \ref{20121217:rem}(c),
since by assumption both $a$ and $b$ are regular,
their right least common multiple is regular too,
and therefore $a_1$ and $b_1$ are regular too.
Multiplying in $Q(R)$ 
both sides of equation \eqref{20121214:eq1} on the left by $u^{-1}$
and on the right by $a_1^{-1}$, we get
\begin{equation}\label{20121214:eq2}
a^{-1}b=(a_1u)^{-1}(1-a_1v)\,,
\end{equation}
and similarly, multiplying \eqref{20121214:eq1} on the left by $v^{-1}$
and on the right by $b_1^{-1}$, we get
\begin{equation}\label{20121214:eq3}
b^{-1}a=(b_1v)^{-1}(1-b_1u)\,.
\end{equation}
Since, by assumption, $a$ and $b$ are left coprime,
both fractions $a^{-1}b$ and $b^{-1}a$ are in their minimal fractional decomposition.
Hence, by equations \eqref{20121214:eq2} and \eqref{20121214:eq3},
there exist $p,q\in R$ such that
\begin{eqnarray}
&& 1-a_1v=pb
\,\,,\,\,\,\,
a_1u=pa
\,,\label{20121218:eq1} \\
&& 1-b_1u=qa
\,\,,\,\,\,\,
b_1v=qb
\,.\label{20121218:eq2}
\end{eqnarray}
Applying the first equation in \eqref{20121218:eq1} to $y\in V$
and using the assumption $ax=by$ and the second equation of \eqref{20121218:eq1},
we get
$$
y=a_1vy+pby=a_1vy+pax=a_1(vy+ux)\,,
$$
and, similarly,
applying the first equation in \eqref{20121218:eq2} to $x\in V$
and using the second equation of \eqref{20121218:eq2},
we get
$$
x=b_1ux+qax=b_1ux+qby=b_1(ux+vy)\,.
$$
Hence, the statement holds with $z=ux+vy$.
\end{proof}

If $V$ is a left $R$-module and $a\in R$, we denote $\ker a=\big\{x\in V\,\big|\,ax=0\big\}$.
\begin{remark}\label{20121217:rem2}
If $d$ is the right greatest common divisor of $a$ and $b$ in $R$, 
then $\ker a\cap\ker b=\ker d$.
Indeed, by the Bezout identity we have $b_1a+a_1b=d$.
Therefore $\ker a\cap\ker b\subset\ker d$.
The reverse inclusion is obvious.
\end{remark}


\begin{corollary}\label{20121214:cor}
Let $R$ be as in Theorem \ref{20121124:thm2},
and let $V$ be a left $R$-module.
Let $\sigma:\, R\to R$ be an anti-automorphism of the ring $ R$.
Let $a,b\in R$, with $b$ regular, be right coprime,
and suppose that $\sigma(a)b=\epsilon\sigma(b)a$,
for some invertible central element $\epsilon\in R$.
Let $x,y\in V$ be such that $\sigma(a)x=\epsilon \sigma(b)y$.
Then there exists $z\in V$ such that $x=bz$ and $y=az$.
\end{corollary}
\begin{proof}
First, since $b$ is regular and $\sigma$ is an anti-automorphism, $\sigma(b)$ is regular too.
Moreover, since by assumption $a$ and $b$ are right coprime
and $\sigma$ is an anti-automorphism,
it follows that $\sigma(a)$ and $\sigma(b)$ are left coprime.

We claim that the left least common multiple of $a$ and $b$
is equal to the right least common multiple of $\sigma(a)$ and $\sigma(b)$,
and it is given by $m=\sigma(a)b=\epsilon\sigma(b)a$.
Indeed, clearly $m$ is a common right multiple of $\sigma(a)$ and $\sigma(b)$.
It is therefore a right multiple of the minimal one: $m_1=\sigma(a)b_1=\sigma(b)a_1$.
Namely, there exists $q\in R$ such that 
$b=b_1q$ and $a=\epsilon^{-1}a_1q$.
But by assumption $a$ and $b$ are right coprime.
Hence, $q$ must be invertible, proving that $m$ is the right least common multiple 
of $\sigma(a)$ and $\sigma (b)$.
The same argument proves that $m$ is also the left least common multiple of $a$ and $b$.

We can now apply Theorem \ref{20121124:thm2}
to $\sigma(a)$ and $\sigma(b)$,
to deduce that there exists $z\in V$ such that $x=bz$ and $\epsilon y=\epsilon az$,
hence $y=az$.
\end{proof}


As in \cite{CDSK12b}, Corollary \ref{20121214:cor} implies the following maximal
isotropicity property important for the theory of Dirac structures, \cite{Dor93,DSK12a}.

\begin{corollary}\label{20121218:cor}
Let $R$ be as in Theorem \ref{20121124:thm2},
and let $V$ be a left $R$-module
and let $(\,\cdot\,,\,\cdot\,):\,V\times V\to A$
be a non-degenerate symmetric bi-additive pairing on $V$
with values in an abelian group $A$.
Let $*$ be an anti-involution of $R$ such that
$(ax,y)=(x,a^*y)$ for all $a\in R$ and $x,y\in V$.
Extend the pairing $(\,\cdot\,,\,\cdot\,)$ to a pairing $\langle\,\cdot\,|\,\cdot\,\rangle$ 
on $V\oplus V$ with values in $A$,
given by
$$
\langle x_1\oplus x_2 \,|\, y_1\oplus y_2\rangle=(x_1,y_2)+(x_2,y_1)\,,
$$
for every $x_1,x_2,y_1,y_2\in V$.
Given two elements $a,b\in R$, we consider the following additive subgroup of $V\oplus V$:
\begin{equation}\label{20121218:eq3}
\mc L_{a,b}=\Big\{bx\oplus ax\,\Big|\,x\in V\Big\}\subset V\oplus V\,.
\end{equation}
\begin{enumerate}[(a)]
\item
The subgroup $\mc L_{a,b}\subset V\oplus V$ is isotropic with respect 
to the pairing $\langle\,\cdot\,|\,\cdot\,\rangle$
if and only if 
$a^*b+b^*a$ acts as $0$ on $V$.
\item
If $b$ is regular,
$a$ and $b$ are right coprime,
and $a^*b+b^*a=0$,
then the subgroup $\mc L_{a,b}\subset V\oplus V$ is maximal isotropic.
\end{enumerate}
\end{corollary}
\begin{proof}
Part (a) is straightforward
and part (b) follows immediately from Corollary \ref{20121214:cor}
with $\sigma(a)=a^*$ and $\epsilon=-1$.
\end{proof}


\begin{corollary}\label{20121214:cor2}
Let $R$ be as in Theorem \ref{20121124:thm2},
and let $V$ be a left $R$-module.
Let $a,b\in R$, with $b$ regular, be left coprime.
Let $m=ab_1=ba_1$ be their right least common multiple.
Then $\ker b=a_1(\ker b_1)$.
\end{corollary}
\begin{proof}
If $k_1\in\ker b_1$, then $b(a_1k_1)=ab_1k_1=0$.
Therefore, $a_1(\ker b_1)\subset\ker b$.
We need to prove the opposite inclusion.
If $k\in\ker b$, we have $a0=0=bk$.
Hence, by Theorem \ref{20121124:thm2},
there exists $z\in V$ such that $0=b_1z$ and $k=a_1z$.
Namely, $k\in a_1(\ker b_1)$.
\end{proof}
\begin{remark}\label{20121214:rem}
In the ring $\mc R=\Mat_{\ell\times\ell}\mc K[\partial]$ of $\ell\times\ell$ matrix differential operators
over a differential field $\mc K$,
the above Corollary \ref{20121214:cor2} implies that
if $b^{-1}a=a_1b_1^{-1}$ is a rational matrix pseudodifferential operator
in its minimal left and right fractional decompositions,
then $\deg(b)=\deg(b_1)$
(where $\deg(b)$ is the degree of the Dieudonn\'e determinant of $b$).
Indeed, the fractional decomposition $b^{-1}a$ being minimal 
means that $a$ and $b$ are left coprime.
Hence, by Corollary \ref{20121214:cor2}
we have that $\dim(\ker b)=\dim(a_1\ker b_1)$
in any differential field extension of $\mc K$.
Moreover, the fractional decomposition $a_1b_1^{-1}$ being minimal 
means that $\ker a_1\cap\ker b_1=0$ in any differential field extension of $\mc K$.
The claim follows by the fact that $\deg b$ is equal to the dimension
of $\ker b$ in the linear closure of $\mc K$ \cite{CDSK12b}.
\end{remark}



\end{document}